\documentclass[]{amsart}
\usepackage{latexsym}
\usepackage{amsmath}
\usepackage{amsfonts}
\usepackage{amsthm}
\usepackage{amssymb}
\usepackage{amsrefs}
\usepackage[T1]{fontenc}

\title[Red-and-black game]{New inequalities for probability functions in the two-person red-and-black game}
\author{W\l{}odzimierz Fechner}
\address{Institute of Mathematics, Lodz University of Technology, ul. W\'olcza\'nska 215, 90-924 \L\'od\'z, Poland}
\email{wlodzimierz.fechner@p.lodz.pl}
\newtheorem{thm}{Theorem}
\newtheorem{cor}{Corollary}
\newtheorem{prop}{Proposition}
\theoremstyle{remark}
\newtheorem{ex}{Example}
\theoremstyle{definition}
\newtheorem*{ack}{Acknowledgements}

\newcommand{\R}{\mathbb{R}}

\newcommand{\Z}{\mathbb{Z}}
\newcommand{\N}{\mathbb{N}}

\newcommand{\f}{\varphi}

\renewcommand{\(}{\left(} \renewcommand{\)}{\right)}


\begin{document}

\begin{abstract}
We discuss a model of a two-person, non-cooperative stochastic game, inspired
by the discrete version of the red-and-black gambling problem presented by Dubins and Savage. 
Assume that two players hold certain amounts of money. At each stage of the game they simultaneously bid some part of their
current fortune and the probability of winning or loosing depends on their bids. In many models of the red-and-black game it is assumed that the win probability is a function of the quotient of the bid of the first player and the sum of both bids. In the literature some additional properties, like concavity or super-multiplicativity are assumed in order to ensure that the bold and timid strategies form a Nash equilibrium, which turns out to be unique in several cases. Our aim is to provide a generalization in which the probability of winning is a two-variable function which depends on both bids. We introduce two new functional inequalities whose solutions lead to win probability functions for which a Nash equilibrium is realized by the bold-timid strategy. Since both inequalities may appear as not quite intuitive, we discuss them in a separate section of the paper and we provide some illustrating examples.
\end{abstract}

\maketitle 

\section{Introduction}

The red-and-black gambling problem goes back to L.E. Dubins and L.J. Savage \cite{DS}. P. Secchi \cite{S} dealt with a two-person red-and-black game in which the first player wants to maximize the probability of winning a certain amount of money, while the second player wants to keep his opponent from this. L. Pontiggia \cites{P1, P2} proposed a modification in which the players' win probabilities can possibly change at each stage of the game. Later, M.-R. Chen and S.-R. Hsiau \cites{Ch1, Ch2}, Y.-C. Yao and M.-R. Chen \cite{Y} provided further modifications and generalizations. The purpose of the paper is to present yet another generalization of the model, in which the win probability is a two-variable function which depends on the two bids.

\medskip

The problem can be described shortly as follows. Two players start the game with some initial fortunes, say $x_0$ and $y_0$ being positive integers. The first player wants to reach a goal, which is an integer $M_1$ greater than $x_0$ and not greater than $x_0+y_0$, and the second player also wants to reach his goal $M_2$, which is an integer between $y_0$ and $x_0+y_0$. Later on we will assume that $M_1=M_2=x_0+y_0$.
The two players bet at each stage of the game an amount  which is a nonnegative integer not greater than their current
fortunes. If a current fortune of the first player is equal to $x$ and he stakes an amount $a$, a current fortune of the second player equals $y$ and he stakes an amount $b$, then their next fortunes will be equal to $x + b$ and $y-b$, respectively with probability $P$, or $x -a $ and $y+a$, respectively with probability $1 - P$. We assume that the probability $P$ depends upon $a$ and $b$.
The game lasts until one of the players reaches his goal. For technical reasons we admit the possibility that $a=0$ or $b=0$. In this case $P$ is equal to $0$ and $1$, respectively, and therefore in case of such bets nothing happens with the player' fortunes (i.e. a bet of zero units gives a sure victory for the opponent, regardless his choice, and of course his gain equals zero units in this case). For the same reasons the value $P(0,0)$ can be undefined.
We also assume that at every stage of the game each player choose his action without any knowledge of the action chosen by the opponent. Thus the game we consider is non-cooperative.

Assume that the bet of player I is equal to $a$ and the bet of player II to $b$. Then the ``fair probability'' $P$ for the first player to win would be ${a}/({a+b})$. In \cites{P1, P2} a class of probabilities under which the game is sub- or superfair is given.
The authors of \cites{Ch1, Ch2} assumed that the probability of winning is a function $f\colon [0,1]\to [0,1]$ which depends upon the value ${a}/({a+b})$. It turned out that convexity and supermultiplicativity of function $f$ 
imply that the bold-timid strategy is a Nash equilibrium for the game. In the present paper we will develop this research further and instead of function $f(a/(a+b) )$ of one variable, we will work with a more general two-place win probability function $P(a,b)$. 

\medskip

The paper is organized as follows. In the next section we formally define the model of the game. First, we check when the game is subfair. Then, we prove that if a certain inequality is satisfied, then an optimal strategy for the first player is to play boldly when the second player plays timidly. Next, we provide conditions under which it is best for the second player to bet timidly when the first player plays boldly. In Section 3 we solve two functional inequalities which are introduced in Section 2. The last section of the paper is devoted to some examples and final comments.

\medskip

Throughout the paper it is assumed that $\R$ denotes the set of real numbers, $\Z$ is the set of integers, $\N $ stands for the set of positive integers and $\N_0 =\N \cup \{0\} $.  
Moreover, for $a, b$ from $\R$ or from $\R\cup \{-\infty, + \infty \}$ respectively, open, closed and half-open intervals with endpoints $a$ and $b$ are denoted by $(a,b)$, $[a,b]$,  $[a,b)$ and $(a,b]$, respectively.

\section{The game}

Assume that we are given an integer $M\geq 2$ which represents the total amount of money in the system. The initial fortune of player I is equal to $x_0 \in \N$ and of player II to $y_0\in \N$. We assume for simplicity that $x_0+y_0=M$ and the goal for both players is to win the amount of $M$. Therefore the state space for both players is the set $S=\{0, 1, \ldots , M\}$ and the game terminates when one of the players reaches $0$ or $M$. The possible actions of a player with fortune $t\in S$ are $\{0, \ldots , t\}$.
A strategy is called \emph{timid}, if at each stage of the game the player bets just one unit of his current wealth.
A strategy is called \emph{bold},  if the player always bets his entire fortune.

Now, let us describe the law of motions in our model. We need a function which assigns the probability of victory as a function of bets of the players. 
 Assume that
$P\colon S \times S \to [0,1]$ is a mapping which satisfies the border conditions:
\begin{equation}\label{P0}
 P(a,0) = 1, \quad  P(0,b)  = 0, \quad a, b \in S\setminus \{0\}.
\end{equation}
The value $P(0,0)$ is left undefined.
As some places we will make a natural assumption that $P$ is nondecreasing with respect to the first variable and nonincreasing with respect to the second one. This means that the chances for winning the turn are better if the player bets (and thus risks) more.
Let $X_n$ denote a random variable which is equal to the fortune of player I at a time $n\in \N$ and suppose he bets an amount $a_n \in \{0, \ldots , X_n\}$. Simultaneously, let the fortune of player II at the same time be denoted by $Y_n$ and assume that he bets $b_n \in \{0, \ldots , Y_n\}$. In our model the laws of motions for the players are defined by formulas
\begin{align}
\label{ini}
X_1&=x_0, \quad Y_1 = y_0,\\ 
\label{lomx}
X_{n+1}&=\left\{  \begin{array}{ll} X_n+b_n, & \textrm{with probability }  P(a_n,b_n), \\
	X_n-a_n, & \textrm{with probability } 1-P(a_n,b_n),
\end{array}            
\right.\\ Y_{n+1}&=M-X_{n+1},\label{lomy}
\end{align}
for $n \in\N$.

We assume that the sets of possible actions of both players contain zero. However, as a rule of the game, they are not allowed to bet zero units whenever their fortunes are positive. Then condition \eqref{P0} means that a hypothetical bet of zero units by any player has no effect on the game (in a sense that the fortunes of both players don't change during this stage).
In particular, from this it follows that $X_{n+1}=X_n$ and $Y_{n+1}=Y_n$ whenever $X_n \in \{0, M\}$. Therefore, we can think of $(X_n)$ and $(Y_n)$ as of infinite sequences.

A special case of this model, when 
\begin{equation}\label{fP}
P(a,b) = f\(\frac{a}{a+b}\), \quad a, b \in S
\end{equation}
with some function $f\colon [0,1]\to [0,1]$ was studied by M.-R. Chen and S.-R. Hsiau in \cite{Ch1}. This in turn is a generalization of L. Pontiggia' models \cites{P1, P2}. 
Our subsequent results can be therefore viewed as an extension of this research. Our approach is motivated by these fine articles. 

The game is called \emph{subfair} if the process ${X_n}$ is a supermartingale under all choices of strategies for the players. Similarly, the game is \emph{superfair} if the process ${X_n}$ is a submartingale under all choices of strategies for both players. We begin with an easy condition upon $P$ which implies that the game is sub- or superfair. This result is a slight modification of the first part of \cite{P1}*{Theorem 3.1}, see also \cite{Ch1}*{p. 907}.

\medskip

\begin{prop}
Assume that $P\colon S \times S \to [0,1]$ is an arbitrary mapping and we consider a discrete two-person red-and-black game with the law of motions described by formulas \eqref{ini}, \eqref{lomx} and \eqref{lomy}. Then the game is subfair  if and only if
$$P(a,b) \leq \frac{a}{a+b}, \quad a, b \in S, \, a+b\leq M.$$
The reverse inequality holds true if and only if the game is superfair.
\end{prop}
\begin{proof}
Fix $n \in \N$. We have 
\begin{align*}
E[X_{n+1}|X_n] &=(X_n+b_n) P(a_n,b_n) + (X_n-a_n)[1-P(a_n,b_n)]\\ &= X_n + (a_n+b_n)\left[P(a_n,b_n) - \frac{a_n}{a_n+b_n}\right].
\end{align*}
Thus, if bids of both players at the $n$-th stage of the game are equal to $a_n$ and $b_n$, respectively, then the sequence $({X_n})$ of random variables is a sub- or supermartingale iff the value in the last square bracket is positive, or negative, respectively.
\end{proof}

Next, we are going to prove a counterpart to \cite{Ch1}*{Theorem 2.1} by M.-R. Chen and S.-R. Hsiau. They showed that the convexity of function $f$ given by \eqref{fP} is a sufficient condition for the fact that the bold strategy is best for player I  if player II plays timidly and the game is subfair. We propose another functional inequality to obtain an analogous effect. However, we do not generalize \cite{Ch1}*{Theorem 2.1}, but instead we provide a new class of win probability functions.

Before we proceed with the theorem, we will recall some crucial notions. We refer the reader to the monograph of A.P. Maitra and W.D. Sudderth \cite{MS} for a comprehensive treatment of the topic. We say that a function $Q\colon S \to \R \cup \{-\infty, + \infty \}$ is \emph{excessive} for some game $\gamma$ at a point $x \in S$ if either $E[Q]$ is undefined or it is defined and $E[Q]\leq Q(x)$. 
A strategy $\sigma$ is a sequence $\sigma_0,  \sigma_1,\dots  $ such that $\sigma_0$ is a gamble and for $n=1, 2, \dots $ by $\sigma_n$ we mean a function which maps a partial history $(x_1, \dots , x_n)$ of the game to a gamble $\sigma_n(x_1, \dots , x_n)$.
Every strategy $\sigma$ determines the corresponding probability measure $P_\sigma$ on the space of all possible results of the game.
If a strategy $\sigma$ of a player is given, then $Q$ is \emph{excessive for $\sigma$} at $x \in S$ if it is excessive at $\sigma_0$ and for every partial history $(x_1, \dots , x_n)$ of positive probability of occurrence $Q$ is excessive for the gamble $\sigma_n(x_1, \dots , x_n)$ at the point $x_n$. The pair $\pi=(\sigma, t)$, where $\sigma$ is a strategy and $t$ is a stopping rule (i.e. stopping time which is everywhere finite), is called a \emph{policy}. By the player's \emph{utility function} we mean every mapping $u\colon S \to \R$. Customarily, some additional assumptions are imposed on $u$. However, for our purposes, later on we will restrict ourselves to the case when $u$ is equal to the identity mapping. Given player's utility $u$, \emph{utility of a policy} is defined as $u(\pi) = E_{\sigma}(X_t)$, i.e. it is the expected utility under probability $P_\sigma$ at a time of stopping. Finally, the policy is \emph{optimal} for the player with the utility function $u$ if its utility is equal to $\sup\limits_\pi u(\pi)$, where the supremum is taken over all available policies.

\medskip

The following theorem provides a criterion for optimality of a strategy in terms of excessivity of some mapping.

\begin{thm}[A.P. Maitra and W.D. Sudderth \cite{MS}, Theorem 3.3.10]\label{Ex}
Given player's utility function $u\colon S \to \R$, a strategy $\sigma$ is optimal if and only if the function $Q\colon S \to \R\cup \{-\infty, + \infty \}$ given by
$$ Q(x) = u(\sigma (x)), \quad x\ \in S $$
is excessive.
\end{thm}

We will work with a one-variable map $\f\colon S \to [0,1]$ defined as 
\begin{equation}\label{fi}
\f(x)= P(x,1), \quad x \in S.
\end{equation}

\begin{thm}\label{t1}
Assume that $P\colon S \times S \to [0,1]$ is a mapping which  satisfies \eqref{P0} and the map  $\f\colon S \to [0,1]$ given by \eqref{fi} is a nondecreasing solution of the inequality
\begin{equation}\label{I1}
\f(y)-\f(x)\leq \f(x-y)[\f(y)-1], 
\end{equation}
for $x, y\in S$ such that $y\leq x$.
We consider a discrete two-person red-and-black game with the law of motions described by formulas \eqref{ini}, \eqref{lomx} and \eqref{lomy}.  
Assume that player II plays a timid strategy. Then the bold strategy is best for player I.
\end{thm}
\begin{proof}
First, we will verify inductively that
\begin{equation}\label{conv}
[1-\f(a)]\prod_{i=0}^a  \f(x-i)\leq {\f(x)-\f(a)}
\end{equation}
for every $x \in \{0, \ldots , M\}$ and $a \in \{0, \ldots , x\}$.
If $a=0$, then clearly $$[1-\f(0)]\f(x) = \f(x)-\f(0),$$
since $\f(0)=0$.
Now, assume that \eqref{conv} holds true for some $a \in S $ and all $x \in \{a, \ldots , M\}$. We will prove it for $a+1$. If $\f(a)=1$, then $\f(a+1)=1$ and there is nothing to prove since the left-hand side equals to zero and $\f$ is nondecreasing. Thus, we assume that $\f(a)<1$. Take some $x\in \{ a+1, \ldots , M\}$; we have
\begin{align*}
[1-\f(a+1)]\prod_{i=0}^{a+1}  \f(x-i) &= \frac{[1-\f(a+1)]}{[1-\f(a)]}[1-\f(a)]\f(x-(a+1)) \prod_{i=0}^{a}  \f(x-i)  \\&\leq \frac{[1-\f(a+1)]}{[1-\f(a)]}  \f(x-(a+1))[\f(x)-\f(a)].
\end{align*}
Therefore, it is enough to show that
$$
[1-\f(a+1)]\f(x-(a+1)) [\f(x)-\f(a)]\leq [1-\f(a)][\f(x)-\f(a+1)].
$$
Since $\f\leq 1$ and $\f$ is nondecreasing, then we are done if we prove
$$[1-\f(a+1)]\f(x-(a+1)) \leq \f(x)-\f(a+1).$$
This in turn follows directly from inequality \eqref{I1} applied for $y= x-(a+1)$.

Since player II plays timidly, then the law of motions of player I who plays boldly takes the form: $X_1=x_0\in S$,
$$
X_{n+1}=\left\{  \begin{array}{ll} X_n+1, & \textrm{with probability } \f(X_n), \\
	0, & \textrm{with probability }  1-\f(X_n),
	\end{array}            
\right.
$$
whenever $0<X_n<M$ and $X_{n+1}=X_n$ if $X_n\in \{0, M\}$.

Denote by $Q(x)$ the probability of victory for player I who has an initial fortune $x$, i.e.
$$Q(x) = \mathbb{P}\{ X_m=M \textrm{ for some } m\in \N \}.$$
It is clear that $Q(0)=0$, $Q(M)=1$ and 
$$Q(x) = \f(x)Q(x+1), \quad x \in \{0, \ldots , M-1\}.$$
Easily one obtains an inductive extension of the above formula:
\begin{equation}\label{eq:}
Q(x-a) = \prod_{i=0}^a \f(x-i)Q(x+1), \quad x \in \{0, \ldots , M-1\}, a \in \{0, \ldots , x\}.
\end{equation}
In view of Theorem \ref{Ex} applied with player's utility function equal to the identity, in order to finish the proof it is enough to prove that function $Q$ is excessive. 
This means that we need to verify inequality
\begin{equation}\label{exc}
\f(a)Q(x+1) + [1-\f(a)]Q(x-a)\leq Q(x)
\end{equation}
for all $x \in \{1, \ldots , M-1\}$ and $a \in \{0, \ldots , x\}$. But this is ensured by \eqref{conv} and \eqref{eq:}.
\end{proof}

Our next statement is a counterpart to \cite{Ch1}*{Theorem 2.2} by M.-R. Chen and S.-R. Hsiau and is in a sense dual to Theorem \ref{t1} above. A new inequality for the two-place win probability function is introduced in order to derive the excessivity condition. 

\begin{thm}\label{t2}
Assume that $P\colon S \times S \to [0,1]$ is a mapping which satisfies \eqref{P0} and
\begin{equation}\label{mult}
P(x,a) P(x+a,b) \leq P(x,a+b)
\end{equation}
 for every $x \in \{0, \ldots , M\}$ and $a \in \{0, \ldots , M-x\}$ and $b \in \{0, \ldots , M-a\}$.
We consider a discrete two-person red-and-black game with the law of motions described by formulas \eqref{ini}, \eqref{lomx} and \eqref{lomy}. Assume that player I plays the bold strategy. Then the timid strategy is best for player II.
\end{thm}
\begin{proof}
Assume that at $n$-th stage of the game player II bets an amount of $b$. If  player I plays boldly, then his law of motions is the form
$$
X_{n+1}=\left\{  \begin{array}{ll} X_n+ b, & \textrm{with probability }  P(X_n,b), \\
	0, & \textrm{with probability }  1-P(X_n,b),
	\end{array}            
\right.
$$
whenever $0<X_n <M$.

Denote by $T(x)$ the probability of victory for player II who has an initial fortune $M-x$. Clearly $T(x)=1-Q(x)$, where $Q$ is defined in the proof of Theorem \ref{t1}. In view of Theorem \ref{Ex} the proof will be completed when we prove that $T$ is excessive. In other words, we need to justify inequality
\begin{equation}\label{excT}
P(x,b)T(x+b) + [1-P(x,b)]T(0)\leq T(x),
\end{equation}
for all $x \in \{0, \ldots , M-1\}$ and $b \in \{0, \ldots , M-x\}$. Directly from the definition we get $T(0)=1$. Therefore, an equivalent formulation of \eqref{excT} with the use of function $Q$ is simply
\begin{equation}
Q(x) \leq P(x,b) Q(x+b).
\label{star}
\end{equation}
From inequality \eqref{mult} we get
$$ P(x,1)P(x+1,1) \leq P(x,2).$$
Using this and \eqref{mult} once again one can prove inductively that
$$ \prod_{i=0}^{b-1}P(x+i,1) \leq P(x,b).$$
Now, using the formula \eqref{eq:} for $a=b-1$ and with $x$ replaced by $x-(b-1)$, jointly with the last inequality we derive
\begin{align*}
Q(x)&= \prod_{i=0}^{b-1}P(x+i,1) Q(x+b) \leq P(x,b) Q(x+b)
\end{align*}
and the proof is completed.
\end{proof}

Joining both theorems we obtain a corollary which shows the existence of a bold-timid Nash equilibrium.

\begin{cor}\label{c}
Assume that $P\colon S \times S \to [0,1]$ is a mapping which  satisfies \eqref{P0} and \eqref{mult} for every $x \in \{0, \ldots , M\}$ and $a \in \{0, \ldots , M-x\}$ and $b \in \{0, \ldots , M-a\}$.
Further, assume that the map  $\f\colon S \to [0,1]$ given by \eqref{fi} is a nondecreasing solution of  \eqref{I1}
for $x, y\in S$ such that $y\leq x$. 
 Then the bold-timid profile is a Nash equilibrium of the discrete two-person red-and-black game with the law of motions described by formulas \eqref{ini}, \eqref{lomx} and \eqref{lomy}.
\end{cor}

A Nash equilibrium postulated by the above statement needs not to be unique since there exist a probability functions which are not strictly monotone on both variables (see Examples \ref{e2} and \ref{el} below). However, if one assumes additionally some natural monotonicity conditions, then the bold-timid profile is the unique Nash equilibrium, like in an analogous result of Pontiggia \cite{P1}*{Corollary 4.1} ).

\begin{cor}\label{c1}
Under assumptions of Corollary \ref{c}, if additionally  the map  $\f\colon S \to [0,1]$ given by \eqref{fi} is strictly increasing and $P(1,y)>0$ for every $y \in S\setminus \{0\}$,
then the bold-timid profile is the unique Nash equilibrium for the game.
\end{cor}
\begin{proof}
We will follow an idea of the proof of \cite{P1}*{Corollary 4.1}. We will use notations of proofs of Theorems \ref{t1} and \ref{t2}.
First, we check that if player II plays timidly, then bold strategy is the unique optimal strategy for player I. More precisely, we will show that the only strategy for which \eqref{exc} holds with an equality is the bold strategy and we have a strict inequality in any other case. Clearly, if $a=x$ in \eqref{exc}, then one obtains an equality. Now, assume that $x \in S$ and $a\in\{1, \dots ,x-1\}$ are fixed. Since $\f$ is strictly increasing, then we have 
$$\f(a)< \f(x) \leq 1, \quad \f(x-a)< \f(x) \leq 1, \quad 0=\f(0)<\f(1).$$
Therefore $\prod\limits_{i=0}^{a+1}  \f(x-i)<1$. Observe also that since $P(1,\cdot)>0$ is positive for positive arguments, then also $Q(x+1)>0$. Now we can estimate the left-hand side of \eqref{exc} as follows:
\begin{align*}
\f(a)Q(x+1) &+ [1-\f(a)]Q(x-a) \\&= \f(a)Q(x+1) +  [1-\f(a)]\prod_{i=0}^{a+1}  \f(x-i)Q(x+1)\\&< ( \f(a) + 1 - \f(a)) Q(x+1) = Q(x+1).
\end{align*}
On the other hand,  the right-hand side of \eqref{exc} equals
$$Q(x)= \f(x)Q(x+1).$$
Eventually, we reach a strict inequality in \eqref{ext}.

Next,  we prove that if player I plays boldly, then the timid strategy is the unique optimal strategy for player II. More precisely, we will show that the only strategy for which \eqref{excT} holds with an equality is the timid strategy and we have a strict inequality in any other case.
We will work with \eqref{star}, which is an equivalent formulation of \eqref{excT}.
Clearly, if $b=1$ in \eqref{star}, then one obtains an equality. Now, assume that $x \in \{ 1, \dots,M-2\}$ and $b\in\{2, \dots ,M-x\}$ are fixed. Arguing as previously and using the fact that $P(x,b)\leq 1$ we estimate the right-hand side of \eqref{star} as follows:
$$
P(x,b)Q(x+b) = P(x,b) \frac{Q(x)}{\prod_{i=1}^{b}  \f(x+b-i)}>Q(x).
$$
The last argument needed in the proof is an application of a one more result proved by Pontiggia \cite{P1}*{Lemma A1}, which guarantees the uniqueness of the equilibrium. 
\end{proof}

\section{Inequalities}

In this section we will construct possibly large classes of functions solving inequalities \eqref{I1} and \eqref{mult}, which play crucial roles in our model. Next section contains examples and a discussion of some particular cases.
We begin with inequality \eqref{mult}. First, we will observe an easy but useful extension property. 

\begin{prop}\label{ext}
Assume that $M \geq 2$ is a fixed integer and $P\colon \{0, \ldots , M\} \times \{0, \ldots , M\} \to [0,1]$ is a mapping which satisfies \eqref{mult}  for every $x \in \{0, \ldots , M\}$ and $a \in \{0, \ldots , M-x\}$ and $b \in \{0, \ldots , M-a\}$. Then a map $\widetilde{P}\colon \Z \times \Z \to [0,1]$ given by 
$$ \widetilde{P} (x,y) = \left\{  \begin{array}{ll}
	0,  & \textrm{if } x< 0 \textrm{ or  } y< 0, \\
	1,  & \textrm{if } x+ y>M, \\ P(x,y), & \textrm{elsewhere} , 
	\end{array}            
\right.
$$
solves \eqref{mult} for all $a, b, x \in \Z$.
\end{prop}
\begin{proof}
Fix $a, b, x \in \Z$ arbitrarily.
Observe that if at least one of numbers $x, a, x+a, b$ is nonpositive, then the left-hand side of \eqref{mult} is equal to zero and \eqref{mult} is satisfied, since $\widetilde{P}$ attains nonnegative values. If $a+b < 0$, then either $a < 0$ or $b <0$ and similarly the left-hand side of \eqref{mult} vanishes.
If $x+a>M$, then \eqref{mult} reduces to a trivial identity.
In other cases \eqref{mult} is satisfied by $\widetilde{P}$ since $P$ solves this inequality.
\end{proof}

Note that it can happen that $\widetilde{P}(x,y)\neq P(x,y)$ for some $(x,y) \in  \{0, \ldots , M\} \times \{0, \ldots , M\}$, but this is possible only in case $x+y>M$, which never happens in the game. 

In view of Proposition \ref{ext} we can study solutions of \eqref{mult} defined on $\Z\times \Z$, or more generally, on the square of an Abelian group.

\begin{prop}\label{Sin}
Assume that $(G, +)$ is an Abelian group and $P\colon G \times G \to \R $ is a mapping which satisfies inequality \eqref{mult} for every $x, a, b \in G$. Then a map $F\colon G \times G \to \R $ given by
\begin{equation}\label{PF}
F(x,a) = P(x,a-x), \quad a, x \in G,
\end{equation}
is a solution of the multiplicative Sincov's inequality:
\begin{equation}
\label{F}
F(x, a ) \cdot F(a, b) \leq F(x,b)  , \quad a , b, x \in G.
\end{equation}
\end{prop}
\begin{proof}
Replace $a$ by $a-x$ in \eqref{mult} to get
$$P(x,a-x)P(a,b+a-a)\leq P(x,a+b-x), \quad a , b, x \in G$$
i.e.
$$F(x,a)F(a,b+a)\leq F(x,a+b), \quad a , b, x \in G.$$
Now, replace in this inequality $b$ by $b-a$ to obtain  \eqref{F}.
\end{proof}

In \cite{ja} we solved inequality \eqref{F} on certain topological spaces. We are going to utilize our results in a special case of $\Z \times \Z$ equipped with the discrete topology. Let us quote \cite{ja}*{Corollary 3}. Assume that $X$ is a nonempty set.
Denote $\Delta = \{(x,x) : x \in X\}$. Next, for a fixed mapping $F\colon X \times X \to (0, + \infty)$  introduce a class of functions:
$$\mathcal{F}(F) = \left\{  f\colon X \to (0, + \infty) : \forall_{ x, y \in X} \frac{f(x)}{f(y)} \geq F(x,y) \right\}.$$
The general solution of \eqref{F} is given in the following statement.

\begin{cor}[\cite{ja}, Corollary 3]\label{cF}
Assume that $X$ is a topological separable space  and  $F\colon X \times X \to (0, + \infty)$ is a solution of \eqref{F} which is continuous and equal to $1$ at every point of $\Delta$. Then
\begin{equation}\label{repF}
F(a,b) = \inf \left\{ \frac{f(a)}{f(b)} : f\in  \mathcal{F}(F) \right\}, \quad a, b \in X.
\end{equation}
Conversely, for an arbitrary family $\mathcal{F}$ of positive functions on $X$  every mapping $F\colon X \times X \to (0, + \infty)$ defined by \eqref{repF} solves \eqref{F}, it is equal to $1$ on $\Delta$ and $\mathcal{F}\subseteq \mathcal{F}(F)$.
\end{cor}

Note that if $F\colon \Z \times \Z \to (0, +\infty)$ is given by \eqref{PF}, then 
$$P(a,b) = F(a,a+b), \quad a, b \in \Z.$$
Moreover, from \eqref{P0} we get $F=1$ on $\Delta$ and $F(0,b)=0$ for $b \in \Z$.
Therefore, if $F $ is positive on $\N \times \N$, then from Corollary \ref{cF} we obtain a representation of $P$:
$$
P(a,b) = \inf \left\{ \frac{f(a)}{f(a+b)} : f\in  \mathcal{F}(F) \right\}, \quad a, b \in \N.
$$
Denote
$$\mathcal{P}(P) = \left\{  f\colon \N_0 \to [0, + \infty) : f(0)=0, \,f(a)>0 \textrm{ for } a > 0, \, 
\forall_{ a, b \in \N} \frac{f(a)}{f(a+b)} \geq P(a,b) \right\}.$$
Therefore we have established the following description of solutions of inequality \eqref{mult} for mappings defined on $\N_0$.

\begin{cor}\label{c2}
Assume that  $P\colon \N_0 \times \N_0 \to [0, + \infty)$ is a solution of \eqref{mult} which  satisfies \eqref{P0} and $P(a,b)>0$ whenever $a>0$. Then 
\begin{equation}\label{repP}
P(a,b) = \inf \left\{ \frac{f(a)}{f(a+b)} : f\in  \mathcal{P}(P) \right\}, \quad a, b \in \N_0, \, a+b>0.
\end{equation}
Conversely, for an arbitrary family $\mathcal{P}$ of functions on $\N_0$ which vanish only at $0$ every mapping $P\colon \N_0 \times \N_0 \to [0, + \infty)$ defined by \eqref{repP} and with $P(0,0)$ being arbitrary solves \eqref{mult}, satisfies \eqref{P0} and $P(a,b)>0$ whenever $a>0$.
\end{cor}

Corollary \ref{c2} provides a method of construction of solutions of \eqref{mult} starting with an arbitrary family of functions.
In order to force the function $P$ given by formula \eqref{repP} to be a win probability function it is enough to assume that $1 \in \mathcal{P}$. Some examples are provided in the next section.

\medskip

Now, we will turn to inequality \eqref{I1}. We will provide a description of solutions in case of differentiable functions defined on an interval. We will need the following fact.

\begin{prop}[\cite{ja2}*{Proposition 1}]\label{pr}
Let $I$ be a non-void open interval, let $f\colon I \to \R$ be a differentiable function and $m \in \R$ an arbitrary constant. Then $$f'(x) \leq m  f(x), \quad x \in I,$$ if and only if there exists a nonincreasing and differentiable map $d\colon I \to \R$ such that $$f(x)=d(x)\exp(mx), \quad x \in I.$$
\end{prop}

For technical reasons we will assume that functions are defined on a half-open interval containing zero as its left endpoint and we work with a one-sided derivative at zero.

\begin{thm}\label{ti}
Assume that $\f\colon [0, M) \to \R$ is a differentiable function with a one-sided derivative at $0$. If $\f$ solves \eqref{I1} for all $x, y \in [0, M]$ such that $y\leq x$, then there exists a nondecreasing and differentiable map $i\colon [0, M) \to \R$ with a one-sided derivative at $0$ which satisfies 
\begin{equation}\label{i}
\f(x)=i(x)\exp(-cx)+1, \quad x \in [0, M),
\end{equation}
where $c=\f'(0)$.
Conversely, if a nondecreasing map $i\colon [0, M) \to \R$ satisfies
\begin{equation}\label{k}
0 \leq i(t)i(y) + i(t+y), \quad t, y \in [0, M), \, t+y \leq M,
\end{equation}
$c \in \R$ is a constant,
then a map $\f\colon [0, M) \to \R$ defined by \eqref{i} solves \eqref{I1}.
\end{thm}
\begin{proof}
We begin with the first part. Fix $t \in (0,M)$ and apply \eqref{I1} for $y=x-t$ to get 
$$\f(x-t) -\f(x) \leq \f(t)[\f(x-t)-1].$$
From this we have
  $$\frac{\f(x-t) -\f(x)}{-t} \geq -\frac{\f(t)}{t}[\f(x-t)-1].$$
Tending with $t \to 0$ and using the differentiability of $\f$ we arrive at
$$
\f'(x) \geq -\f'(0)[\f(x)-1], \quad x\in (0,M).
$$
Define $f\colon [0, M) \to [0,1]$ as $f(x)=\f(x)-1$ for $x\in [0,M)$ and put $c= \f'(0)$. Then we have
\begin{equation}\label{diff}
f'(x) \geq -cf(x), \quad x\in (0,M).
\end{equation}
To finish this part of the proof it is enough to apply Proposition \ref{pr} with $f$ replaced by $-f$ (we have $m=-c$ and $i=-d$). 

A straightforward calculation shows that for a given map $i\colon [0, M) \to \R$ function $\f\colon [0, M) \to \R$ given by \eqref{i} with arbitrary $c\in \R$ is a solution to \eqref{I1} if and only if $$0 \leq i(x-y)i(y) + i(x), \quad x, y \in [0, M), \, y \leq x,$$ which is true by \eqref{k} after substitution $t=x-y$.
\end{proof}

\section{Examples and final remarks}

First, we will show how a known example of win probability can be obtained as a special case of our model with the aid of Corollary \ref{c2}.  

\begin{ex}
Assume that we are given some number $p \geq 1$. Define $\mathcal{P}$ as a family consisting of exactly one function $f\colon \N_0 \to [0, + \infty)$ given by $f(t) = t^p$ for $t \in \N_0$. Then $P(a,b) = a^p/(a+b)^p$. This win probability was introduced by M.-R. Chen and S.-R. Hsiau  \cite{Ch1}*{Example 2.2 and Example 2.4}. Note that in particular \eqref{I1} and \eqref{mult} are satisfied.
\end{ex}

An curious example is provided by a family $\mathcal{P}$ consisting of a linear map and an exponential. 

\begin{ex}\label{e2}
Let $m \in (0,\infty)$ be a fixed number.
Define $f_i\colon \N_0 \to [0, \infty)$ for $i\in \{1, 2\}$ as
$$f_1(k) = k, \quad f_2(k) = \exp \(mk\), \quad k \in \N$$ and $f_i(0)=0$. Therefore, for $a, b \in \N_0$ we have 
$$P(a,b) = \min \left\{ \frac{a}{a+b}, \exp\(-mb\) \right\}.$$ 
Next, observe that 
 $$\frac{a}{a+b} = \frac{1}{1+\frac{b}{a}}\geq \frac{1}{\exp (\frac{b}{a})} = \exp \(-\frac{b}{a}\)$$
for $a, b \in \N$ such that  $a>1/m$. Thus $P(a,b) = \exp\(-mb\)$, in particular in this case the probability does not depend upon $a$. If player I plays boldly and his initial fortune is large enough to meet the condition $x_0>1/m$  and player 2 plays timidly, then the probability of winning a stage of the game for the first player is equal to $\exp\(-m\)$, which can be arbitrarily close to zero if $m$ is large enough. Moreover, it is clear that it is not best for player 1 to play boldly if his fortune exceeds the number $\lceil 1/m \rceil$, since he can bet precisely $\lceil 1/m \rceil$ to obtain the same possible gain as playing boldly, but with a smaller risk. This however does not contradict Theorem \ref{t1}, since in this case $\f$ given by \eqref{fi} is a constant map equal to $\exp(-m)$ and this is not a solution of \eqref{I1}. From this example we also see that without additional assumptions upon a win probability function a Nash equilibrium postulated Corollary \ref{c} needs not to be unique. 
\end{ex}

With the aid of Theorem \ref{ti} one can construct a family of functions which solve \eqref{I1}. 
To obtain a win probability function one needs to find a nondecreasing function $i$ which satisfies \eqref{k} and such that  $\f$ given by \eqref{i}  takes values in the interval $[0,1]$ (for the latter reason, one cannot take simply an arbitrary nonnegative nondecreasing function as $i$). Denote $k=-i$. Clearly, $i$ solves \eqref{k} if and only if $k$ satisfies:
\begin{equation}\label{sm}
 k(t+y) \leq k(t)k(y), \quad t, y \in [0, M), \, t+y \leq M.
\end{equation}
Observe that if $k$ takes values in the interval $[0,1]$ and the constant $c$ is nonnegative, then the map $\f$ given by \eqref{i} with $i=-k$ takes values in $[0,1]$, as well. We can summarize this short argument in the following statement.

\begin{cor}
Assume that $c \geq 0$ is a constant and $k\colon [0,M)\to [0,1]$ is a nonincreasing map which satisfies \eqref{sm}. Then function 
$\f\colon [0,M) \to [0,1]$ given by
\begin{equation}\label{kp}
\f(x)=-k(x) e^{-cx} +1, \quad x\in [0,M)
\end{equation}
yields a nondecreasing solution of \eqref{I1}. 
\end{cor}

For a discussion of inequality \eqref{sm} the reader is referred to the monograph E. Hille and R.S. Phillips \cite{HP} (inequality (7.4.1) therein).
Below we will examine the simplest case with $k=1$.

\begin{ex}\label{el}
Assume that $M\in \N$ and the win probability function $P\colon [0,M] \times [0,M] \to [0,1]$ is given by
$$P(a,b) = -e^{b-a}+1, \quad a, b \in [0,M], \, b \leq a$$
and $P(a,b) = 0$ if $b>a$.
One can check that the game is neither sub- nor superfair. Function $\f$ given by \eqref{fi} solves \eqref{I1} and is represented by formula \eqref{kp} with $k(x) = 1$.
Thus, all conditions of Theorem \ref{t1} are met. Note that
$P(x,x+1)=0$ for all $x \in S$; in other words player II wins every time he bets more than his opponent. This does not contradict Theorem \ref{t1}, since we assume that player II plays timidly. Therefore, in case when the initial fortunes of the two players satisfy $x_0<y_0$, the bold-timid strategy is not a unique Nash equilibrium, since player II playing boldly wins for sure, regardless the strategy of his opponent. 
\end{ex}

\bigskip

\begin{ack}
The author would like to express his most sincere gratitude to the three anonymous reviewers for a number of valuable comments regarding the previous versions of the manuscript which have led to the essential improvement of the whole paper. 
\end{ack}

\end{document}